\documentclass[11pt]{amsart}
\usepackage{amsmath,amssymb,mathrsfs,pstricks}
\usepackage[small,heads=LaTeX,nohug]{diagrams}
\usepackage[colorlinks=true,linkcolor=black,citecolor=black,urlcolor=black]{hyperref}

\psset{unit=1pt}
\psset{arrowsize=4pt 1}
\psset{linewidth=.5pt}

\newcommand{\excise}[1]{}

\newcommand{\isom}{\cong}
\renewcommand{\setminus}{\smallsetminus}
\renewcommand{\phi}{\varphi}

\renewcommand{\hat}{\widehat}
\renewcommand{\bar}{\overline}

\newcommand{\GG}{\mathbb{G}}
\newcommand{\PP}{\mathbb{P}}
\newcommand{\QQ}{\mathbb{Q}}

\newcommand{\ZZ}{\mathbb{Z}}

\newcommand{\shfF}{\mathscr{F}}

\newcommand{\ee}{\mathrm{e}}  
\newcommand{\eem}{\varepsilon} 

\newcommand{\kk}{k}    

\newcommand{\OO}{\mathcal{O}} 

\newcommand{\id}{\mathrm{id}}
\newcommand{\pt}{\mathrm{pt}}

\newcommand{\op}{\mathrm{op}}
\newcommand{\opk}{{\op}K}

\newcommand{\loc}{\mathrm{loc}} 

\newcommand{\perf}{\mathrm{perf}}
\newcommand{\VB}{\mathrm{vb}}

\newcommand{\td}{\mathrm{td}} 

\DeclareMathOperator{\Hom}{Hom}

\DeclareMathOperator{\Sym}{Sym}

\DeclareMathOperator{\Tor}{Tor}
\DeclareMathOperator{\PExp}{PExp}

\renewcommand{\div}{\mathrm{div}}

\newtheorem{theorem}{Theorem}[section]

\newtheorem{proposition}[theorem]{Proposition}
\newtheorem{corollary}[theorem]{Corollary}

\theoremstyle{definition}

\begin{document}

\title{Computing torus-equivariant K-theory of singular varieties}

\date{March 31, 2016}

\author{Dave Anderson}
\address{Department of Mathematics, The Ohio State University, Columbus, OH 43210}
\email{anderson.2804@math.osu.edu}
\thanks{This work was partially supported by NSF DMS-1502201.}



\maketitle

\section{Introduction}

Vector bundles on algebraic varieties are basic objects of study.  Among the many questions one can ask, some fundamental ones are these: What are the global sections of a vector bundle $E$ on a variety $X$?  How can they be computed?  Does $X$ carry any nontrivial vector bundles at all?

Somewhat more tractable than the space of global sections $H^0(X,E)$ is the Euler characteristic $\chi(X,E) := \sum (-1)^i \dim(H^i(X,E))$, which makes sense whenever these dimensions are finite---e.g., when $X$ is complete.  This function is additive on short exact sequences, so one is led to consider the {\em Grothendieck group of vector bundles},
\[
  K^\circ_\VB(X) := \left\langle [E] \,\Big|\, [E]=[E']+[E''] \text{ whenever } 0 \to E' \to E \to E'' \to 0 \right\rangle,
\]
i.e., the free abelian group on isomorphism classes of vector bundles, modulo the given relation for each short exact sequence.  The Euler characteristic thus defines a function $K^\circ_\VB(X) \to \ZZ$, when $X$ is complete.

If $X$ is nonsingular and comes with an action of an algebraic group---say, a torus $T = (\GG_m)^n$---then one can often take advantage of the group action to simplify many calculations, including Euler characteristics.  Indeed, there is a version of the Atiyah-Bott localization formula in this context: assuming for simplicity that $X$ has finitely many $T$-fixed points,
\begin{equation}\label{e.loc-intro}
  \chi_T(X,E) = \sum_{p\in X^T} \frac{ [E_p] }{ (1-[L_1(p)^*])\cdots (1-[L_d(p)^*]) }.
\end{equation}
Here, for any finite-dimensional $T$-representation $V$, $[V]$ denotes its graded character, or equivalently, its class in the representation ring $R(T)$.  In the numerator on the right-hand side, $E_p$ is the fiber of the equivariant vector bundle $E$ at the fixed point $p$; in the denominator, the $L_i$'s form a decomposition of the tangent space $T_pX$ into one-dimensional weight spaces for the $T$-action.  On the left-hand side, the equivariant Euler characteristic is defined as $\chi_T(X,E) := \sum (-1)^i [H^i(X,E)]$ in $R(T)$.  By forgetting the $T$-action and remembering only dimension, one gets a homomorphism $R(T) \to \ZZ$ which takes $\chi_T$ to $\chi$.

The localization formula \eqref{e.loc-intro} thus reduces the computation of Euler characteristics to a finite calculation, and one which is often quite easy.  As a toy example, consider $T=\GG_m$ acting on $\PP^1$ by $z\cdot [a,\,b] = [a,\,z b]$.  For $n\in \ZZ$, let us write $\ee^{nt} \in R(T)$ for the one-dimensional representation given by $z\cdot v = z^n v$, for $z\in T$.  The action on $\PP^1$ has two fixed points, $0=[1,0]$ and $\infty=[0,1]$, and it induces a natural action on the line bundle $E=\OO(1)$, so that $[E_0]=\ee^t$ and $E_\infty = \ee^0 = 1$.  The tangent spaces are $[T_0\PP^1]=\ee^t$ and $[T_\infty\PP^1] = \ee^{-t}$, so their duals are $[L(0)^*] = \ee^{-t}$ and $[L(\infty)^*]=\ee^t$.  Putting all this into the formula, we have
\begin{align*}
  \chi_T(\PP^1,\OO(1)) &= \frac{ \ee^t }{1-\ee^{-t}} + \frac{ 1 }{1-\ee^{t}} \\
                       &= 1 + \ee^t,
\end{align*}
which, taking $t\mapsto 0$, recovers the familiar fact that $\OO(1)$ has two sections (once one knows it has vanishing $H^1$).

When $X$ is singular, $K_{\VB}^\circ(X)$ may be rather complicated, even for relatively simple varieties, such as complete toric varieties.  The combinatorial structure of such varieties makes many of their invariants finite and computable, yet there are examples of (projective, three-dimensional) toric varieties $X$ such that $K_\VB^\circ(X)$ contains a copy of the ground field (and, in particular, may be uncountably generated) \cite{gubeladze}.  On the other hand, on a general (singular, non-projective) toric variety, it is not known if there are any non-trivial vector bundles at all.

The purpose of this note is to survey some basic properties and applications of {\em operational $K$-theory}, especially as applied to varieties with the action of a (split) torus.  The groups $\opk_T^\circ(X)$ (and the non-equivariant groups, $\opk^\circ(X)$) are defined rather abstractly, but they turn out to be more computable than the ``geometric'' theories $K_\perf^\circ(X)$ and $K_\VB^\circ(X)$.  Furthermore, any Euler characteristic computation, such as the one exhibited above, can be carried out operationally: there will be canonical maps $K_\VB^\circ(X) \to K_\perf^\circ(X) \to \opk^\circ(X)$, and when $X$ is complete, the Euler characteristic $\chi\colon K_\VB^\circ(X) \to \ZZ$ factors though these homomorphisms.

Operational cohomology was introduced by Fulton and MacPherson \cite{fm} to serve as a contravariant counterpart to Chow homology groups, since no other option was available.  For $K$-theory, there are already several contravariant counterparts to the ``homology'' $K$-theory of coherent sheaves.  They all map to $\opk^\circ$, so it is natural to study this theory as well.  We will also see that properties of operational $K$-theory yield applications to ordinary $K$-theory: Corollary~\ref{c.surjective} implies that the usual $K$-theory of a complete toric threefold is always nontrivial, and Proposition~\ref{p.not-surj} gives an example of a projective toric variety which has $K$-theory classes not lifting to equivariant classes.

\medskip
\noindent
{\it Acknowledgements.}  I am grateful to Michel Brion and Mahir Can for organizing the Clifford Lectures at Tulane University.  I also thank my collaborators on this project, Richard Gonzales and Sam Payne, and the referee for helpful comments on the manuscript.

\section{Some background}

All schemes will be separated and of finite type over an algebraically closed field $\kk$.  A torus $T$ has character group $M=\Hom_{\mathrm{alg. gp.}}(T,\GG_m) \isom \ZZ^n$.

\subsection{Equivariant coherent sheaves}

When $T$ acts on a scheme $X$, a $T$-action on a coherent sheaf $\shfF$ is defined as follows: writing $a\colon T \times X \to X$ for the action map, and $p\colon T\times X \to X$ for the projection, one must specify isomorphisms $a^*\shfF \isom p^*\shfF$, satisfying some natural compatibilities.  (See, e.g., \cite[\S5]{cg}.)  A coherent sheaf equipped with a $T$-action in this way is called an {\it equivariant coherent sheaf}.  A homomorphism of equivariant coherent sheaves is one that respects the actions.  We will write $(T\text{-}\mathbf{Coh}_X)$ for the resulting abelian category of $T$-equivariant coherent sheaves on $X$.

The $K$-theory of equivariant coherent sheaves, $K^T_\circ(X)$, is defined to be the Grothendieck group of $(T\text{-}\mathbf{Coh}_X)$.

\subsection{Equivariant Chow groups}

The {\it Chow group} $A_i(X)$ is defined as $i$-dimensional cycles modulo rational equivalence: $Z_i(X)$ is the free abelian group on $i$-dimensional subvarieties of $X$, and $R_i(X)$ is the subgroup generated by cycles of the form $[\div_W(f)]$, for $f$ a rational function on some $(i+1)$-dimensional subvariety $W$.  For schemes equipped with the action of an algebraic group, Edidin and Graham defined {\it equivariant Chow groups} $A^G_*(X)$ to be Chow groups of quotients constructed from Totaro's approximations to the classifying space \cite{eg-eit,totaro-eq}.  If the scheme $X$ is smooth, these groups fit together to form a graded ring under intersection product; one often uses cohomological grading and writes $A_G^*(X)$ in this case.

When the group is a torus $T$, as it will be here, Brion gave a concrete characterization of Edidin-Graham's equivariant Chow groups.  Let
\[
  \Lambda = \Lambda_T = \Sym_\ZZ^* M \isom \ZZ[t_1,\ldots,t_n];
\]
it is a basic fact that $A^T_*(\pt) = A_T^*(\pt) = \Lambda$.

\begin{theorem}[{\cite[Theorem~2.1]{brion-chow}}]
  The equivariant Chow group $A^T_*(X)$ is identified with the $\Lambda$-module generated by $[Y]$ for $Y\subseteq X$ a $T$-invariant subvariety, subject to relations $[\div_W(f)]-\lambda\cdot[W]$, for $W$ an invariant subvariety and $f$ a rational function on $W$, which is an eigenfunction of weight $\lambda \in M$.
\end{theorem}

\subsection{Vector bundles and perfect complexes}\label{ss.tech}

In order to have a $K$-theory for which one knows good local-to-global properties, like Mayer-Vietoris and localization sequences, one has to work not with Grothendieck groups of vector bundles, but of perfect complexes.  A {\em perfect complex} of $\OO_X$-modules is one which is locally quasi-isomorphic to a bounded complex of vector bundles.  This notion was introduced in \cite{sga6}, and the basic properties of their $K$-theory were established in the remarkable paper \cite{tt}.  (The equivariant analogues of these foundational results have been recently developed in \cite{krishna-ravi}.)  Because it possesses the expected properties in general, while $K^\circ_\VB$ does not, one usually writes $K^\circ(X) := K^\circ_\perf(X)$.  Here we will generally preserve the subscripts for clarity---the exception being when we have restricted to quasi-projective varieties, for in this case it is known that $K_\VB = K_\perf$.  In fact, this isomorphism is known somewhat more generally, e.g., whenever $X$ admits an embedding in a smooth variety.

The question of the relationship between vector bundles and perfect complexes remains open for general schemes.  
Certainly any vector bundle is a perfect complex, and there is a natural homomorphism $K^\circ_\VB(X) \to K_\perf^\circ(X)$. 
However, local quasi-isomorphisms need not glue to a global isomorphism, so the nature of this homomorphism is generally ill understood; in fact, there are simple examples of non-separated schemes for which the two $K$-theories are distinct.  Implicit in \cite[\S3]{tt} is the fact that $K^\circ_\VB(X) \to K_\perf^\circ(X)$ is an isomorphism whenever $X$ has the {\em resolution property}: every coherent sheaf on $X$ is the quotient of a vector bundle.  This is made explicit in \cite{resolution}, where the resolution property is discussed in more detail; see also \cite{payne-vbs}, which deals with the case of toric varieties. 

The reader who is content to restrict to quasi-projective schemes may safely ignore this technical point about perfect complexes.  On the other hand, there are singular, non-projective, complete toric threefolds which, until very recently, were not known to carry any nontrivial vector bundles at all; but it follows from results we shall discuss later that $K^\circ_\perf$ is nontrivial for every such variety.  (Perling and Schr\"oer recently proved that every complete toric threefold carries vector bundles with arbitrarily large third Chern class, so in fact, they have many nontrivial vector bundles \cite{ps}.)

The takeaway of this remark is the following: the main results to be described here do not directly imply anything about vector bundles on a scheme, except in cases where vector bundles are known to exist for independent reasons.

\section{Bivariant theories}

To provide a framework for analyzing and proving Riemann-Roch type theorems, Fulton and MacPherson introduced the notion of a {\em bivariant theory}.  Fix a class of fiber squares of schemes, called ``independent squares'', which includes all squares where one side is the identity.  To each morphism of schemes $X \to Y$, the theory assigns a graded abelian group $B(X \to Y)$, along with three basic operations:
\begin{enumerate}
\item Given a composition of morphisms $X \to Y \to Z$, there is a {\em product} homomorphism
\[
   B(X\to Y) \otimes B(Y \to Z) \xrightarrow{\cdot} B(X \to Z).
\]
\item If $f\colon X\to Y$ is proper, and $g\colon Y \to Z$ is any morphism, there is a {\em pushforward} homomorphism
\[
  B(X \to Z) \xrightarrow{f_*} B(Y \to Z).
\]
\item For any morphism $g\colon Y' \to Y$ such that the resulting fiber square
\begin{diagram}
  X' & \rTo^{f'} & Y' \\
 \dTo &         & \dTo_g \\
 X  &  \rTo^f   & Y
\end{diagram}
is independent, there is a {\em pullback} homomorphism
\[
  g^*\colon B(X \to Y) \to B(X' \to Y').
\]
\end{enumerate}
Particular cases are the associated ``homology'' theory $B_*(X) = B(X \to \pt)$, which is covariant (via pushforward) for proper morphisms; and the associated ``cohomology'' theory, $B^*(X) = B(\id \colon X \to X)$, which is contravariant (via pullback) for arbitrary morphisms, and is a ring under the product operation. 
The three operations are required to satisfy a number of axioms, which we will not list here; for complete descriptions, see the original \cite{fm} or the summaries in \cite{ap,gk}.

Elements of a bivariant group $B(X \to Y)$ can be understood as generalized Gysin homomorphisms, and from this point of view, the axioms model the usual behavior of Gysin maps.  Indeed, $\alpha \in B(f\colon X \to Y)$ defines a ``wrong-way'' Gysin pullback $f^\alpha\colon B_*(Y) \to B_*(X)$, by
\[
  f^\alpha(y) = \alpha\cdot y ,
\]
and if $f$ is proper, also a Gysin pushforward $f_\alpha\colon B^*(X) \to B^*(Y)$, by
\[
 f_\alpha(x) = f_*( x\cdot \alpha ).
\]

There are three main examples of bivariant theories that play a role in the present story: operational Chow theory, operational $K$-theory, and bivariant $K$-theory of perfect complexes.

\subsection*{Relatively perfect complexes}

Suppose $X$ and $Y$ are quasi-projective.  Then any morphism $f\colon X \to Y$ factors as a closed embedding $\iota\colon X \hookrightarrow P$, followed by a smooth projection $p\colon P\to Y$.  An {\em $f$-perfect complex} of sheaves on $X$ is $\shfF^\bullet$ such that $\iota_*\shfF^\bullet$ is quasi-isomorphic to a bounded complex of vector bundles on $P$.  Defining $K_\perf^\circ(X\to Y)$ to be the Grothendieck group of $f$-perfect complexes on $X$ yields a bivariant theory, whose independent squares are {\em Tor-independent} \cite{fm}.  (A fiber square
\begin{diagram}
  X' & \rTo & Y' \\
 \dTo &  & \dTo \\
 X & \rTo & Y
\end{diagram}
is Tor-independent if $\Tor^Y_i(\OO_X,\OO_{Y'})=0$ for all $i>0$.)

\subsection*{Operational theories}

Looking for a cohomological counterpart to Chow homology, Fulton and MacPherson defined an {\em operational bivariant theory} by imitating the relationship between singular (Borel-Moore) homology and cohomology in topology.  In topology, if $g\colon Y \to X$ is any continuous map, then an element $c\in H^i(X)$ acts as a homomorphism $H_j(Y) \to H_{j-i}(Y)$, sending $\alpha\in H_j(Y)$ to $g^*(c) \cap \alpha$.  Proceeding backward from this property, for a scheme $X$, a class in operational Chow cohomology $c\in A^i(X)$ is defined to be a collection of homomorphisms $c_g \colon A_j(Y) \to A_{j-i}(Y)$, one for each morphism $g\colon Y \to X$; these are required to satisfy some basic compatibilities (projection formula, etc.), modelled on the topological situation.

For the rest of this section, we will work with the category of schemes with $T$-action (``$T$-schemes'') and equivariant morphisms.  The definitions and properties of operational Chow and $K$-theory are quite similar, so we will focus on the latter.  The purpose here is to give a general impression of these bivariant theories; complete definitions can be found in \cite{ap}.

The independent squares for operational $K$-theory are all fiber squares.  A class $c\in \opk_T^\circ(X \to Y)$ is a collection of homomorphisms $c_g \colon K^T_\circ(Y') \to K^T_\circ(X')$, one for each morphism $g\colon Y' \to Y$ (with $X' = X \times_Y Y'$).  Product is given by composition of homomorphisms.  In addition to compatibility with pushforward and pullback, we require that classes in $\opk_T^\circ(X\to Y)$ {\em commute} with Gysin pullbacks for flat maps and regular embeddings.

To see what this means, suppose $h\colon W \to Z$ is a regular embedding, and $Y' \to Z$ is any morphism, with $Y'' = Y' \times_Z W$.  One can define a homomorphism $h^!\colon K^T_\circ(Y') \to K^T_\circ(Y'')$ by sending $[\shfF]$ to $\sum (-1)^i [\Tor_i^{\OO_Z}(\shfF,\OO_W)]$, a finite sum since for regular embeddings $\Tor^i$ vanishes for large enough $i$.  When $h$ is flat, it is even easier to define such a homomorphism: all higher Tor vanishes, so $h^![\shfF] =[\shfF\otimes_{\OO_Z}\OO_W]$.  Commuting with Gysin pullbacks means that in Figure~\ref{f.diags}, the diagram of fiber squares on the left produces a commutative diagram on the right.

\begin{figure}
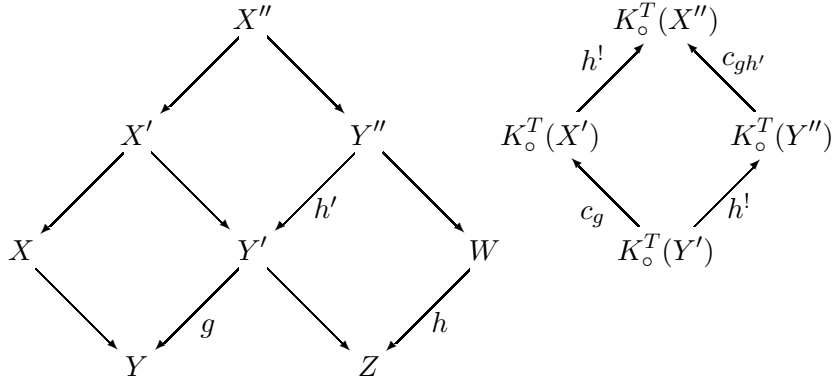

\begin{diagram}
   &   &   &       &  X''     &  &  & & & &  &  K^T_\circ(X'') \\
   &   &   & \ldTo &    & \rdTo  &  & & & & \ruTo^{h^!} &  & \luTo^{c_{gh'}} \\
   &   & X' &     &     &    & Y''  & &  & K^T_\circ(X') & & & & K^T_\circ(Y'') \\
   & \ldTo & & \rdTo &  & \ldTo_{h'} & & \rdTo &  & & \luTo_{c_g} & & \ruTo_{h^!} \\
 X &   &    &     & Y'  &       &  &   & W &  & & K^T_\circ(Y') \\
  & \rdTo & & \ldTo_g &  & \rdTo  & & \ldTo_h \\
  &     & Y &      &   &      & Z  
\end{diagram}
\caption{Commuting with Gysin pullback, for $h$ flat or regular embedding.\label{f.diags}}
\end{figure}

If $X$ is an $n$-dimensional variety, there are always homomorphisms $A_T^i(X) \to A^T_{n-i}(X)$ and $\opk_T^\circ(X) \to K^T_\circ(X)$, sending an operator to its value on $[X]$ or $[\OO_X]$, respectively.  An important property of operational groups is that these are ``Poincar\'e isomorphisms'' when $X$ is smooth:

\begin{proposition}[{\cite[Proposition~17.4.2]{fulton-it},\cite[Proposition~4]{eg-eit},\cite[Proposition~4.3]{ap}}] \label{p.poincare}
Suppose $X$ is smooth.  The natural homomorphisms $A_T^i(X) \to A^T_{n-i}(X)$ and $\opk_T^\circ(X) \to K^T_\circ(X)$ are isomorphisms.
\end{proposition}

More generally, if $f\colon X \to Y$ is any morphism, and $g\colon Y \to Z$ is smooth, then there is a distinguished ``orientation'' class $[g]\in \opk_T^\circ(Y \to Z)$ giving rise to a natural isomorphism
\[
  \opk_T^\circ(X \to Y) \xrightarrow{\cdot [g]} \opk_T^\circ(X \to Z).
\]
(Taking $X = Y$, $Z=\pt$, and $f=\id$, one recovers Proposition~\ref{p.poincare}.)  A similar statement holds for bivariant Chow theory.

When $X$ is complete, composing with the equivariant Euler characteristic defines a homomorphism
\begin{align*}
  \opk_T^\circ(X) &\to \Hom_{R(T)}(K^T_\circ(X), R(T)), \\
       c &\mapsto (\alpha \mapsto \chi_T( c_{\id}(\alpha) ) ).
\end{align*}
In general, this is neither injective nor surjective.  However, when $X$ is a {\em $T$-linear variety}---a class which includes toric varieties, spherical varieties, and Schubert varieties---it is an isomorphism.  This is an echo of similar statements for Chow cohomology (\cite{fmss,totaro}).

\begin{theorem}[{\cite[Theorem~6.1]{ap}}]\label{t.kronecker}
For a complete $T$-linear variety $X$, we have $\opk_T^\circ(X) \isom \Hom_{R(T)}(K^T_\circ(X), R(T))$.
\end{theorem}

\noindent
In particular, since $K^T_\circ(X)$ is a finitely generated $R(T)$-module for such varieties, so is $\opk_T^\circ(X)$.  This stands in contrast to $K^\circ(X)$, which may contain a copy of the base field, and thus be uncountably generated \cite{gubeladze}.

\section{Kimura's exact sequences and the Kan extension property}

The definitions of operational Chow and $K$-theory may appear unwieldy: a class in $\opk_T^\circ(X)$ is a collection of compatible endomorphisms of $K^T_\circ(Y)$, as $Y$ ranges over all schemes mapping equivariantly to $X$.  However, the main tool for computing---introduced for (non-equivariant) Chow theory by Kimura and developed for equivariant $K$-theory in \cite{ap}---turns out to be quite effective, especially when one has access to resolution of singularities.

An {\em equivariant envelope} is a $T$-equivariant proper map $X' \to X$ such that every invariant subvariety of $X$ is the birational image of an invariant subvariety of $X'$.

\begin{proposition}[{\cite[Theorem~2.3]{kimura},\cite[Proposition~5.3]{ap}}]
Let $X' \to X$ be an equivariant envelope.  The {\em first Kimura sequences}
\begin{align}
 & 0 \to A_T^*(X) \to A_T^*(X') \to A^*(X'\times_X X') &  \text{and}\\ 
 & 0 \to \opk_T^\circ(X) \to \opk_T^\circ(X') \to \opk_T^\circ(X'\times_X X')
\end{align}
are exact.
\end{proposition}

\noindent
The key ingredients in the proof are the corresponding exact sequences for homology theories,
\begin{align*}
 & A^T_*(X' \times_X X') \to A^T_*(X') \to A^T_*(X) \to 0 & \text{and} \\
 & K^T_\circ(X' \times_X X') \to K^T_\circ(X') \to K^T_\circ(X) \to 0.
\end{align*}
These were established (in the nonequivariant case) by Kimura \cite{kimura} and Fulton-Gillet \cite{fg,gillet}, respectively, and extended to the equivariant setting in \cite[Appendix]{ap}.

When the envelope is birational, a more precise statement is available.  An {\em abstract blowup diagram} is a fiber square
\begin{diagram}
 E & \rInto & X' \\
\dTo &  & \dTo \\
 S & \rInto & X,
\end{diagram}
with $X' \to X$ proper, $S\subseteq X$ closed, and $X'\setminus E \to X \setminus S$ an isomorphism.

\begin{proposition}[{\cite[Theorem~3.1]{kimura},\cite[Proposition~5.4]{ap}}]
Given an abstract blowup square as above, suppose additionally that $X' \to X$ is an equivariant envelope.  Then the {\em second Kimura sequences}
\begin{align}
 & 0 \to A_T^*(X) \to A_T^*(X') \oplus A_T^*(S) \to A_T^*(E) &  \text{and} \\
 & 0 \to \opk_T^\circ(X) \to \opk_T^\circ(X') \oplus \opk_T^\circ(S) \to \opk_T^\circ(E) \label{e.kimura2k}
\end{align}
are exact.
\end{proposition}

Combined with the ``Poincar\'e isomorphisms'' of Proposition~\ref{p.poincare}, the second Kimura sequence allows computations of operational theories to be carried out by reduction to the smooth case.  
A useful consequence is the {\em Kan extension} property.  In a precise sense, $\opk_T^\circ(X)$ is the universal target for contravariant ``cohomology'' theories on schemes, which when restricted to smooth schemes, agree with $K_T^\circ(X)$.  Regarding $K_T^\circ$ as a functor
\[
  K_T^\circ \colon (T\text{-}\mathbf{Sm})^{\mathrm{op}} \to (R(T)\text{-}\mathbf{Mod}),
\]
this is the statement that equivariant operational $K$-theory is the (right) Kan extension of $K_T^\circ$ along the inclusion $(T\text{-}\mathbf{Sm})^{\mathrm{op}} \hookrightarrow (T\text{-}\mathbf{Sch})^{\mathrm{op}}$ of smooth schemes in all schemes.  (This is a basic notion in category theory; in fact, most familiar constructions---limits, colimits, etc.---can be realized as certain Kan extensions.  See \cite{cwm}.)  Explicitly:

\begin{proposition}[{\cite[Theorem~5.8]{ap}}]
Assume the base field has characteristic zero.  Let $L_T^\circ$ be any contravariant functor from $T$-schemes to $R(T)$-modules, whose restriction to smooth schemes admits a natural transformation to $K_T^\circ$.  Then there is a unique extension of this transformation to a transformation $\eta\colon L_T^\circ \to \opk_T^\circ$.
\end{proposition}

The same proof works, {\em mutatis mutandis}, to show that in characteristic zero (or whenever there exist suitable resolutions of singularities), equivariant operational Chow cohomology $A_T^*(X)$ is the right Kan extension of the equivariant intersection ring on smooth $T$-varieties.  
It follows that there is a natural homomorphism
\[
  ``A_T"(X) := \lim_{X \to Y} A_T^*(Y) \to A_T^*(X),
\]
where the limit in the source is taken over all maps of $X$ to smooth schemes $Y$.  (This construction was offered as a substitute for the intersection ring in \cite{fulton75}, prior to the introduction of the operational Chow ring.  In fact, $``A"(X)$ is the {\em left} Kan extension of the intersection ring, essentially by definition.)

More interestingly, defining $``K_T"(X)$ as the analogous limit, there are natural homomorphisms
\begin{equation}\label{e.ks}
  ``K_T"(X) \to K_{T,\VB}^\circ(X) \to K_{T,\perf}^\circ(X) \to KH_T^\circ(X) \to \opk_T^\circ(X),
\end{equation}
all of which are isomorphisms when $X$ is smooth.  
Here $KH_T^\circ(X)$ is (the degree-zero part of) Weibel's homotopy $K$-theory (see \cite{weibel})---or more precisely, its equivariant version, constructed very recently in \cite{krishna-ravi}.

Non-equivariant homotopy $K$-theory possesses a descent property for abstract blowup squares ({\em cdh-descent}, proved in \cite{haesemeyer}), which in particular implies that there is a natural exact sequence
\[
  KH^1(E) \to KH^0(X) \to KH^0(X') \oplus KH^0(S) \to KH^0(E).
\]
Combining this with the second Kimura sequence \eqref{e.kimura2k} and some basic facts about toric varieties, one can prove something about the rightmost map of \eqref{e.ks} in the non-equivariant case:

\begin{theorem}[{\cite[Theorem~7.1]{ap}}]\label{t.surjective}
If $X$ is a three-dimensional toric variety, then the natural homomorphism $KH^\circ(X) \to \opk^\circ(X)$ is surjective.
\end{theorem}

For general toric varieties, it is proved in \cite{chww} that the map $K_\perf^\circ(X) \to KH^\circ(X)$ is a split surjection.  This, together with Theorem~\ref{t.kronecker}, lets us deduce that $K^\circ_\perf(X)$ is nontrivial for complete toric threefolds; more specifically:

\begin{corollary}[{\cite[Theorem~1.4]{ap}}]\label{c.surjective}
For any complete three-dimensional toric variety, the homomorphism $K^\circ_\perf(X) \to \Hom(K_\circ(X),\ZZ)$ is surjective.
\end{corollary}

\noindent
At the time this was proved, we did not know whether all complete toric varieties admit nontrivial vector bundles, even in dimension three (see \cite{payne-vbs} for a discussion of this question).  Results of Gharib and Karu \cite{gharib-karu} gave conditions on the fan of a toric threefold guaranteeing nontrivial $K^\circ_\perf(X)$ for many such varieties; the corollary above extends this fact to all complete toric threefolds.  Almost concurrently, Perling and Schr\"oer \cite{ps} proved that in fact complete toric threefolds do carry nontrivial vector bundles, but their methods, like ours, seem to run into difficulties in higher dimension.

It would be very interesting to know cdh-descent for equivariant homotopy $K$-theory.  As pointed out in \cite{krishna-ravi}, this would have many applications; one might be the equivariant analogue of Corollary~\ref{c.surjective}.

\section{Riemann-Roch theorems}

As mentioned before, the original motivation for introducing bivariant theories was to unify several Riemann-Roch type theorems.  For this section, we restrict to the category of quasi-projective schemes.  In this context, Fulton \cite[\S18]{fulton-it} showed how to construct natural homomorphisms
\[
  K^\circ(X \to Y) \xrightarrow{\tau} A^*(X \to Y)
\]
which induce commuting squares
\begin{equation}\label{d.verdier}
\begin{diagram}
  K_\circ(X)  & \rTo^\tau & A_*(X)_\QQ \\
   \uTo^{f^!} &   & \uTo_{\td(T_f)\cdot f^!} \\
  K_\circ(Y)  & \rTo^\tau & A_*(Y)_\QQ
\end{diagram}
\end{equation}
when $f$ is smooth, and
\begin{equation}\label{d.sga6}
\begin{diagram}
  K^\circ(X)  & \rTo^\tau & A^*(X)_\QQ \\
   \dTo^{f_!} &   & \dTo_{f_!( \; \cdot \td(T_f))} \\
  K^\circ(Y)  & \rTo^\tau & A^*(Y)_\QQ
\end{diagram}
\end{equation}
when $f$ is smooth and proper.  (We often write $A_*(X)_\QQ$ for $A_*(X)\otimes \QQ$.)  Here $T_f$ is the relative tangent bundle of $f$, and $\td$ is the {\em Todd class}. It can be characterized formally by setting
\[
  \td(L) = \frac{x}{1-\ee^{-x}}
\]
for a line bundle $L$ with first Chern class $x$, and requiring the property $\td(E) = \td(E')\cdot\td(E'')$ whenever $0\to E' \to E \to E'' \to 0$ is an exact sequence of vector bundles.

When $Y$ is a point, the homomorphism $f_!$ is identified with the Euler characteristic, and the second diagram expresses the classical Hirzebruch-Riemann-Roch formula.

The Riemann-Roch transformation $\tau$ was constructed in the equivariant setting by Edidin and Graham \cite{eg-rr}.  Here one must also take certain completions: let $\hat{A}^T_*(X) = \prod_{i\in \ZZ} A^T_i(X)$, and let $\hat{K}^T_\circ(X)$ be the completion with respect to the kernel of the homomorphism $R(T) \to \ZZ$ sending each $\ee^t$ to $1$.

\begin{theorem}[{\cite{eg-rr}}]\label{t.eg}
There is a natural homomorphism $\tau\colon K^T_\circ(X) \to \hat{A}^T_*(X)_\QQ$, which induces an isomorphism $\hat{K}^T_\circ(X)_\QQ \xrightarrow{\sim} \hat{A}^T_*(X)_\QQ$.
\end{theorem}

This can be extended to the bivariant setting:

\begin{theorem}[{\cite{agp}}]\label{t.brr}
There are natural homomorphisms $\opk_T^\circ(X \to Y) \xrightarrow{t} \hat{A}_T^*(X \to Y)_\QQ$, inducing isomorphisms $\hat{\opk^\circ_T}(X \to Y)_\QQ \xrightarrow{t} \hat{A}_T^*(X \to Y)_\QQ$.  These commute with the forgetful homomorphisms $\opk_T^\circ \to \opk^\circ$ and $A_T^* \to A^*$.
\end{theorem}

\noindent
The proof is mostly formal, building on Theorem~\ref{t.eg}.  In the non-equivariant case, this factors the perfect complex transformation: for quasi-projective schemes, Fulton's argument in fact constructs a transformation $K^\circ \to \opk^\circ$, such that the composition
\[
  K^\circ(X \to Y) \to \opk^\circ(X\to Y) \xrightarrow{t} A^*(X \to Y)_\QQ
\]
is the transformation $\tau$.  It would be interesting to know if this can be extended to general schemes.

The Riemann-Roch theorems can be used to exhibit an example of a $3$-dimensional projective toric variety $X$ such that $K^\circ_T(X) \to K^\circ(X)$ is not surjective.  This should be viewed as contrasting with $K^T_\circ(X) \to K_\circ(X)$, which is surjective for any variety \cite{merkurjev}.

\begin{proposition}[{\cite{agp}}]\label{p.not-surj}
Let $X$ be the toric mirror dual to $(\PP^1)^3$, i.e., corresponding to the fan over the faces of a cube.  Then $K_T^\circ(X) \to K^\circ(X)$ is not surjective.
\end{proposition}

\begin{proof}
It is shown in \cite{kp} that $A_T^*(X)_\QQ \to A^*(X)_\QQ$ is not surjective, so neither is $\alpha\colon \hat{A}_T^*(X)_\QQ \to A^*(X)_\QQ$.  Now consider the diagram
\begin{diagram}
  K^\circ_T(X) & \rTo & \opk_T^\circ(X) & \rInto & \hat{\opk_T^\circ}(X)_\QQ &\rTo^\sim & \hat{A}_T^*(X)_\QQ \\
  \dTo^\gamma  &   & \dTo   &      & \dTo   &  & \dTo_\alpha \\
  K^\circ(X) & \rTo^\beta & \opk^\circ(X) & \rInto & \opk^\circ(X)_\QQ &\rTo^\sim & {A}^*(X)_\QQ.
\end{diagram}
The homomorphism $\beta\colon K^\circ(X) \to \opk^\circ(X)$ is surjective by Theorem~\ref{t.surjective}, so a diagram chase shows that $\gamma$ cannot be.
\end{proof}

\section{Localization theorems}

One of the most useful features of equivariant $K$- and Chow theory is the possibility of computing by localizing at $T$-fixed points.  At the foundation of this technique are isomorphisms
\[
  \bar{S}^{-1}A_*^T(X^T) \xrightarrow{\sim} \bar{S}^{-1}A_*^T(X)
\]
and
\[
  S^{-1}K_\circ^T(X^T) \xrightarrow{\sim} S^{-1}K^T_\circ(X),
\]
where $\bar{S} \subseteq \Lambda_T$ is the multiplicative set generated by all nonzero $\lambda\in M$, and $S\subseteq  R(T)$ is generated by elements of the form $1-\ee^{-\lambda}$, for nonzero $\lambda\in M$.  (These isomorphisms were established in \cite[\S2.3, Corollary 2]{brion-chow} and \cite[Th\'eor\`eme~2.1]{thomason-lef}, respectively.)  These can be extended to the bivariant setting:

\begin{theorem}[{\cite{agp}}]
There are natural homomorphisms
\begin{align*}
  \bar{S}^{-1} A_T^*(X \to Y) &\xrightarrow{\loc^A} \bar{S}^{-1} A_T^*(X^T \to Y^T) 
 \quad \text{and}\\
  S^{-1}\opk_T^\circ(X\to Y) &\xrightarrow{\loc^K}  S^{-1} \opk_T^\circ(X^T \to Y^T),  
\end{align*}
inducing isomorphisms of $\bar{S}^{-1}\Lambda_T$-modules and $S^{-1}R(T)$-modules, respectively, and commuting with the basic bivariant operations (product, pushforward, and pullback).
\end{theorem}

This bivariant version of the localization theorem is formally similar to the Riemann-Roch theorem (Theorem~\ref{t.brr}); in fact, both are deduced from a general statement about transformations of operational bivariant theories.  Origins of this formal similarity can be found in the Lefschetz-Riemann-Roch theorem of Baum, Fulton, and Quart \cite{bfq}.

The Riemann-Roch formulas expressed by diagrams \eqref{d.verdier} and \eqref{d.sga6} have localization analogues.  Suppose $f\colon X \to Y$ is a proper flat equivariant map, with the induced map $\bar{f}\colon X^T \to Y^T$ also flat.  In $K$-theory, one asks for classes $\eem^K(f)\in\opk_T^\circ(X^T)$ making the diagrams
\begin{equation}\label{d.loc1}
\begin{diagram}
  S^{-1}K^T_\circ(X)  & \rTo^{\loc^K} & S^{-1}K^T_\circ(X^T) \\
   \uTo^{f^!} &   & \uTo_{\eem^K(f)\cdot \bar{f}^!} \\
 S^{-1} K^T_\circ(Y)  & \rTo^{\loc^K} & S^{-1}K^T_\circ(Y^T)
\end{diagram}
\end{equation}
and
\begin{equation}\label{d.loc2}
\begin{diagram}
  S^{-1}\opk_T^\circ(X)  & \rTo^{\loc^K} & S^{-1}\opk_T^\circ(X^T) \\
   \dTo^{f_!} &   & \dTo_{\bar{f}_!( \; \cdot \eem^K(f))} \\
  S^{-1}\opk_T^\circ(Y)  & \rTo^{\loc^K} & S^{-1}\opk_T^\circ(Y^T),
\end{diagram}
\end{equation}
commute.  (In Chow theory, one has the corresponding problem of finding classes $\eem^A(f)\in A_T^*(X^T)$.)

\begin{theorem}[{\cite{agp}}]
In the above setting, if $\bar{f}\colon X^T \to Y^T$ is smooth, there exist unique classes $\eem^K(f)$ and $\eem^A(f)$ fitting into commutative diagrams as depicted.
\end{theorem}

The classes $\eem^K(f)$ and $\eem^A(f)$ are called ($K$- or Chow-theoretic) {\em total equivariant multiplicities} of $f$.  To justify the name, let us consider the case where $Y=\pt$ and $X^T$ consists of finitely many {\it nondegenerate} fixed points, meaning that for each $p\in X^T$, the zero character does not occur among the weights of the $T$-action on $T_pX$.  (In this situation, the scheme-theoretic fixed locus $X^T$ is reduced, so $\bar{f}$ is  smooth; see, e.g., \cite[Proposition~A.8.10(2)]{cgp}.)  Let us write $\eem_p^K(X)$ and $\eem_p^A(X)$ for the restrictions of the total multiplicities to $S^{-1}R(T)=S^{-1}\opk_T^\circ(p)$ and $\bar{S}^{-1}\Lambda_T = \bar{S}^{-1}A_T^*(p)$, respectively.

\begin{proposition}[{\cite{agp}}]
Let $\lambda_1,\ldots,\lambda_n$ be the weights of $T$ acting on the Zariski tangent space $T_pX$, and let $C=C_pX \subseteq T_pX$ be the tangent cone.  Then
\[
  \eem^K_p(X) = \frac{[\OO_C]}{(1-\ee^{-\lambda_1})\cdots(1-\ee^{-\lambda_n})} \quad \text{ and } \quad \eem^A_p(X) = \frac{[C]}{\lambda_1 \cdots \lambda_n},
\]
as elements of $S^{-1}R(T) = S^{-1}K_T^\circ(T_pX)$ and $\bar{S}^{-1}\Lambda_T = \bar{S}^{-1}A_T^*(T_pX)$, respectively.

In particular, $\eem^A_p(X)$ is the Brion-Rossmann equivariant multiplicity of $X$ at $p$ \cite{brion-chow}.
\end{proposition}

Continuing this basic setup, so $Y=\pt$ and all fixed points of $X$ are nondegenerate, diagram \eqref{d.loc2} expresses an Atiyah-Bott-type localization formula:

\begin{corollary}
Given a $T$-equivariant vector bundle $E$ of rank $r$ on $X$, its equivariant Euler characteristic may be computed as
\[
  \chi_T(E) = \sum_{p\in X^T}  (\ee^{\chi_1(p)}+\cdots+\ee^{\chi_r(p)}) \cdot \eem_p^K(X),
\]
where the fiber of $E$ is a $T$-representation with weights $\chi_1(p),\ldots,\chi_r(p)$.
\end{corollary}

\noindent
The computation done in the introduction is a simple special case.

The total equivariant multiplicities play a role in localization analogous to that of the Todd class in Riemann-Roch.  In fact, they are directly related, at least in the case where $X$ and $Y$ are smooth, with finitely many fixed points: for each $p\in X^T$, one has
\[
  \td(T_f)|_{p} = \frac{\eem^K_p(f)}{\eem^A_p(f)}
\]
in an appropriate localization of $\hat{R(T)} = \hat\Lambda_T$.

\section{Other directions}

%
\subsection{Chang-Skjelbred and GKM theorems}

Part of the initial motivation for introducing operational $K$-theory was to find a geometric interpretation for the ring of {\em piecewise exponential functions} on the fan $\Delta$ associated to a toric variety $X=X(\Delta)$.  An example of such a function is shown in Figure~\ref{f.p112}; the  fan has three maximal (two-dimensional) cones, with the southwest-pointing ray passing through $(-1,-2)$.  The corresponding toric variety is the weighted projective space $\PP(1,1,2)$.

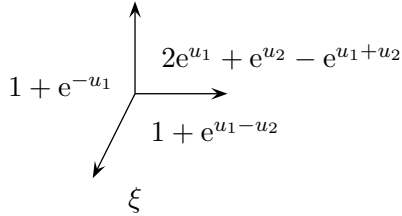
\begin{figure}
\begin{pspicture}(-35,-50)(100,50)

\psline{->}(0,0)(35,0)
\psline{->}(0,0)(0,35)
\psline{->}(0,0)(-16,-32)

\rput[l](10,14){$2\ee^{u_1}+\ee^{u_2}-\ee^{u_1+u_2}$}
\rput[l](6,-14){$1+\ee^{u_1-u_2}$}
\rput[r](-8,3){$1+\ee^{-u_1}$}

\rput(0,-40){$\xi$}

\end{pspicture}
\caption{A piecewise exponential function for $\PP(1,1,2)$ \label{f.p112}}
\end{figure}

In \cite{vv}, it was shown that when $X=X(\Delta)$ is nonsingular, $K_T^\circ(X)$ is isomorphic to the ring $\PExp(\Delta)$. (A similar result was proved in \cite{bv} for more general simplicial toric varieties, taking rational coefficients.  For equivariant Chow cohomology, analogous results were proved by Brion in the smooth case \cite{brion-chow}, and Payne in the general case \cite{payne-chow}.)  For general $X(\Delta)$, one should take operational $K$-theory on the geometric side: there is always an isomorphism $\opk_T^\circ(X) \isom \PExp(\Delta)$ \cite[Theorem~1.6]{ap}.  A similar result was proved by Harada-Holm-Ray-Williams for the topological equivariant $K$-theory and cobordism rings of weighted projective spaces, under some conditions on the weights \cite{hhrw}.

This can be regarded as an instance of a Chang-Skjelbred or GKM-type theorem.  Over fields of characteristic zero, R.~Gonzales has shown that this phenomenon is quite general in operational $K$-theory.  A {\em $T$-skeletal variety} is one such that both $X^T$ and the set of one-dimensional $T$-orbits are finite.

\begin{theorem}[{\cite[Theorem~5.4]{gonzales}}]
For a complete $T$-skeletal variety $X$, $\opk_T^\circ(X) \to \opk_T^\circ(X^T)$ induces an isomorphism onto the subring
\[
  \PExp(X) := \{ (f_p)_{p\in X^T} \,|\, f_p-f_q \text{ is divisible by } (1-\ee^{\chi_{p,q}})\};
\]
here $\chi_{p,q}$ is the character of the one-dimensional orbit connecting fixed points $p$ and $q$.
\end{theorem}

Furthermore, if $X$ is any complete $T$-variety, the restriction homomorphism $\opk_T^\circ(X) \to \opk_T^\circ(X^T)$ is injective \cite[Proposition~3.7]{gonzales}.

\subsection{Bivariant algebraic cobordism}

The algebraic cobordism theory $\Omega_*(X)$ of Levine and Morel acts as a covariant ``homology'' theory with respect to proper maps.  J.~Gonz\'alez and K.~Karu have defined a corresponding equivariant operational bivariant theory, $\Omega_T^*$, developed its properties in order to compute for toric varieties: $\Omega_T^*(X(\Delta))$ is isomorphic to a ring of piecewise graded power series on $\Delta$ \cite[Theorem~7.3]{gk}.

It would be interesting to know more about the relation of $\Omega_T^*$ with $\opk_T^\circ$ and $A_T^*$; for instance, one might look for a Riemann-Roch type transformation from $\Omega_T^*$ to all other such bivariant theories.



\end{document}